\documentclass[10pt]{amsart}
\usepackage[margin=1.2in,marginparsep=0.1in,marginparwidth=1in]{geometry}
\usepackage{amssymb,amsmath,amsthm,amstext,amscd,latexsym,graphics,graphicx,bbm,caption}
\usepackage[usenames,dvipsnames,svgnames,table]{xcolor}
\usepackage[plainpages=false,colorlinks=true, pagebackref]{hyperref}
\usepackage{tikz}
\usetikzlibrary{positioning}
\tikzset{>=stealth}
\hypersetup{citecolor=Sepia,linkcolor=blue, urlcolor=blue}

\usepackage{cite}   %upright parenthesis for cite

\makeatletter
\def\@tocline#1#2#3#4#5#6#7{\relax
  \ifnum #1>\c@tocdepth % then omit
  \else
    \par \addpenalty\@secpenalty\addvspace{#2}%
    \begingroup \hyphenpenalty\@M
    \@ifempty{#4}{%
      \@tempdima\csname r@tocindent\number#1\endcsname\relax
    }{%
      \@tempdima#4\relax
    }%
    \parindent\z@ \leftskip#3\relax \advance\leftskip\@tempdima\relax
    \rightskip\@pnumwidth plus4em \parfillskip-\@pnumwidth
    #5\leavevmode\hskip-\@tempdima
      \ifcase #1
       \or\or \hskip 2em \or \hskip 2em \else \hskip 3em \fi%
      #6\nobreak\relax
    \dotfill\hbox to\@pnumwidth{\@tocpagenum{#7}}\par
    \nobreak
    \endgroup
  \fi}
\makeatother
\newtheorem{intro-thm}{Theorem}[]
\theoremstyle{plain}
\newtheorem{thm}{Theorem}[section]
\newtheorem{theorem}[thm]{Theorem}

\newtheorem{lemma}[thm]{Lemma}

\theoremstyle{definition}
\newtheorem{remark}[thm]{Remark}

\newcommand{\ilim}{\mathop{\varprojlim}\limits} % inverse limit
\newcommand{\inj}{\hookrightarrow}

\newcommand{\F}{{\mathbb F}}

\newcommand{\N}{{\mathbb N}}

\newcommand{\Z}{{\mathbb Z}}

\let\syn\mathsf

\newcommand{\ca}{\frac{A}{[A,A]}}

\input{xy}
\xyoption{all}
\begin{document}

\title[Morphisms between two constructions of Witt vectors of non-commutative rings.]
{Morphisms between two constructions of Witt vectors of non-commutative rings } 

\author[S. Pisolkar]{Supriya Pisolkar} \address{ Indian Institute of
Science, Education and Research (IISER),  Homi Bhabha Road, Pashan,
Pune - 411008, India} \email{supriya@iiserpune.ac.in} 

\date{}
\begin{abstract} 
Let $A$ be any unital associative, possibly non-commutative ring and let $p$ be a prime number.  Let $E(A)$ be the ring of $p$-typical Witt vectors as constructed by Cuntz and Deninger in \cite{cd} and $W(A)$ be the abelian group constructed by Hesselholt in \cite{h1} and \cite{h2}.  In \cite{hp} it was proved that if $ p=2$ and $A$ is non commutative unital torsion free ring then there is no surjective continuous group homomorphism from $W(A) \to HH_0(E(A)): = E(A)/\overline{[E(A),E(A)]}$ which commutes with the  Verschiebung  operator and the Teichm\"uller map. In this paper we generalise this result to all primes $p$ and  simplify the arguments used for $p=2$.  We also prove that if $A$ a is non-commutative unital ring then there is no continuous map of sets $HH_0(E(A)) \to W(A)$ which commutes with the ghost maps. 
 \end{abstract}
\maketitle 
%%%%%%%%%%%%%%%%%%%%%%%%%%%%%%%%%%%%%%%%%%%
\section{Introduction}
Let $p$ be a prime number. Let $A$ be any unital associative, non-commutative ring. In \cite{hp} we compared  two constructions, one of a ring $E(A)$ by Cuntz and Deninger, given in \cite{cd} and the other of an abelian group $W(A)$ given by Hesselholt in \cite{h1} (see also \cite{h2}). Both $E(A)$ and $W(A)$ are topological groups and are equipped with the Verschiebung  operator $V$ and the Teichm\"uller map $\langle \cdot \rangle$. Moreover,  $W(A)$ and $E(A)$ are isomorphic to the classical construction of ring of $p$-typical Witt vectors when $A$ is commutative. It is natural to see how these constructions are related when $A$ is non-commutative. L. Hesselholt asked whether for a associative ring $A$,  $W(A)$ is isomorphic to $HH_0(E(A))$? Although this question is still open, it was proved in \cite[Theorem 1.2]{hp} that for $p=2$ and $A= \Z\{X,Y\}$ there is no continuous surjective group homomorphism from $W(A) \to HH_0(E(A))$ which commutes with $V$ and $\langle \cdot \rangle$. One of the main results of this paper generalises this result to any prime number $p$.  

\begin{theorem}\label{generalise}
Let $A:=\Z\{X,Y\}$ and $p$ be any prime number. Then there is no continuous surjective group homomorphism from $W(A) \to HH_0(E(A))$ which is compatible with $V$ and $\langle \ \rangle$. 
\end{theorem}

It is also natural to see whether there is a map in the opposite direction giving relation between $HH_0(E(A))$ and $W(A)$. The next result of this paper will show that  even this is not possible under some additional hypothesis in the case  when $p$ is any prime number and $A= \Z\{X,Y\}$.  

\begin{theorem} \label{main} Let $p$ be any prime number. Let $A = \Z\{X, Y\}$. Then there is no map of sets  from $HH_0(E(A)) \to W(A)$ which commutes with the ghost maps $\overline{\eta}: HH_0(E(A)) \xrightarrow{ \ \overline{\eta} \ } \Big(\ca\Big)^{\N_0}$ and $ W(A) \xrightarrow{\ \overline{\omega}\ } \Big( \ca \Big)^{\N_0}.$
\end{theorem}
\vspace{2mm}

%%%%%%%%%%%%%%%%%%%%%%%%%%%%%%%%%%%%%%%%%%%%%
\noindent{\bf Acknowledgement}:  The author is indebted to the anonymous referee for providing many insightful comments. This work is supported by the SERB-MATRICS grant MTR/2018/000346. \\

%%%%%%%%%%%%%%%%%%%%%%%%%%%%%%%%%%%%%%%%%%%%
\section{Preliminaries}

\noindent In this section we will briefly recall the constructions $W(A)$ from \cite{h1},\cite{h2} and of $E(A)$ from \cite{cd}. \\

\noindent{\bf (1) Hesselholt's construction of  $W(A)$ :}\label{hesselholt-construction} \\

\noindent We will stick to the hypothesis on $A$ as in \cite{h2}. Suppose $A$ is any unital associative ( need not be commutative)  ring $A$. Let $p$ be a prime number and $\N_0 := \N \cup \{0\}$. \\

\noindent Consider the map (called as ghost map) 
$$ \omega: A^{\N_0} \to \Big( \ca \Big)^{\N_0} $$  
$$ \omega(a_0,a_1,a_2,...) := \big(\omega_0(a_0), \omega_1(a_0,a_1), \omega_2(a_0,a_1,a_2),...\big)$$
where $\omega_i$'s are ghost polynomials defined by 
$$\omega_i(a_0,...,a_i) := a_0^{p^i}+pa_1^{p^{i-1}}+p^2a_2^{p^{i-2}}+\cdots+p^i a_i.$$
$\omega$ is merely a map of sets and not a homomorphism of groups. 
For every integer $n\in \N_0$, we also have truncated versions of the above map (denoted again by $\omega$)

\noindent Hesselholt then inductively defines groups $W_n(A)$ (see \cite{h2}) such that 
the map $\omega$ factor through 
$$A^n \xrightarrow{q_n} W_n(A) \xrightarrow{\overline{\omega}} \Big(\ca\Big)^n$$ 
and the following are satisfied 
\begin{enumerate} 
\item $W_1(A) = \ca$. 
\item $q_n$ is surjective map of sets. 
\item $\overline{\omega}$ is an additive homomorphism and is injective if $\ca$ is $p$-torsion free.
\end{enumerate}
\vspace{2mm}

\noindent Define $W(A) := \ilim_n W_n(A)$ and the topology on $W(A)$ is the inverse limit topology. Clearly one also has a factorization of $A^{\N_0}\xrightarrow{\omega} \big(\ca\big)^{\N_0}$ as 
$$ A^{\N_0}\xrightarrow{q} W(A) \xrightarrow{\overline{\omega}} \Big( \ca \Big)^{\N_0}$$
where $q$ is always surjective and where $\overline{\omega}$ is injective if $\ca$ has no $p$-torsion. %Note that, in \cite{h2},  the definition of $W(A)$ for a general ring $A$ is given by using the fact that, given a general ring $A$ we can define a surjective ring homomorphism from $A' \to A$ where $ \frac{A'}{[A',A'] }$ is $p$-torsion free.

\noindent We  have the Verschiebung operator  
$$V: W(A) \to W(A) $$ 
and the Teichm\"uller map 
$$ \langle \ \rangle : A \longrightarrow W(A) $$
which satisfy

$$V (a_0,a_1,\cdots) = (0,a_0,a_1,\cdots)$$
 and
$$\langle a\rangle= ( a,0,0,\ldots)$$

\noindent One can show that $V$ and $\langle \ \rangle$ are well defined and that $V$ is a additive group homomorphism. Similarly for $n\in \N_0$, we have truncated versions (denoted by the same notation). \\

\vspace{2mm}
\noindent{\bf Ghost map :} The group homomorphism $\overline{\omega} : W(A) \to \Big( \ca \Big)^{\N_0}$ given by 
$$\overline{\omega}(a_0,a_1,a_2,...) := \big(\omega_0(a_0), \omega_1(a_0,a_1), \omega_2(a_0,a_1,a_2),...\big)$$
will also be called as the ghost map and $\overline{\omega}$ is injective if $\frac{A}{[A,A]}$ is $p$-torsion free ( See \cite{h2}, page 56 ).  \\ 

%%%%%%%%%%%%%%%%%%%%%%%%%%%%%%%%%%%%%%%%%%%%%
\vspace{2mm}

\noindent{\bf (2) Cuntz and Deninger's construction of the ring $E(A)$ :}\\

\noindent  Let $A$   be any associative, possibly non-unital ring $A$, $p$ be a prime number and $\N_0 := \N \cup \{0\}$. We will refer to \cite[Preliminaries and Page 20]{cd}.\\

\noindent Consider the ring $A^{\N_0}$ with the product topology where $A$ has the discrete topology. 
\begin{enumerate}
\item[(i)]  Let $V: A^{\N_0} \to A^{\N_0}$ be the map defined by  
$V(a_0,a_1,... ) := p(0,a_0,a_1,....)$.
\item[(ii)] For an element $a\in A$, define $ \langle a \rangle \in A^{\N_0}$ by 
$ \langle a \rangle := (a,a^p,a^{p^2},...)$.
\item[(iii)] Let $X(A)\subset A^{\N_0} $ be the closed subgroup generated by 
$$ \Big\{ V^m(\langle a_1 \rangle \cdots  \langle a_r \rangle ) \ | \ m \in \N_0, r\in \N, \ a_i \in A  \ \forall \ i\Big\}.$$
\noindent Similarly, if $I\subset A$ is an ideal, we let $X(I)$ denote the closed subgroup generated by 
$$ \Big\{ V^m(\langle a_1 \rangle \cdots  \langle a_r \rangle ) \ | \ m  \in \N_0, r\in \N, \ a_i \in I \ \forall \ i \Big\}.$$
\end{enumerate}

\noindent For $n\in \N_0$, we also have the  truncated version $X_n(A), n \in \N$ ( See Preliminaries in \cite{cd}). In fact $X(A) =  \ilim_n X_n(A)$ as topological rings. \\

\noindent Let $\Z A$ be the monoid algebra of the multiplicative monoid underlying $A$. Thus the elements of $\Z A$ are formal sums of the form $ \sum_{r\in \Z A} n_r [r] \ \ \ \text{with almost all }n_r=0$. We have a natural epimorphism of rings from $\Z A \to A$ and we let $I$ denote its kernel. One now defines 
$$ E(A) := \frac{X(\Z A)}{X(I)}  \ \ \ \text{and} \ \ \ E_n(A):= \frac{X_n(\Z A)}{X_n(I)} $$ 
Note that $E(A)$ is a Hausdorff topological ring equipped with the ( multiplicative) Teichm\"uller map $\langle \rangle$ and the continuous additive operator $V$ give by  
$$ \langle \ \rangle : A \to E(A) \ \ \ \ \  \langle a \rangle := ([a],[a]^p,[a]^{p^2},...) \ \ \syn{mod} \ X(I).$$
$$V: E(A) \to E(A) \  \ \ \ \  V(a_0,a_1,...,a_n ) := p(0,a_0,a_1,...,a_{n-1})$$
\vspace{2mm}

\noindent The above construction gives a functor $E$ from the category of associative rings to the category of associative rings which is compatible with the map $\langle \ \rangle $ and the additive homomorphism $V$.  \\

%%%%%%%%%%%%%%%%%%%%%%%%%%%%%%%%%%%%%%%%%%%%
\noindent {\bf  Ghost maps :}  Let $X(A) \xrightarrow{\ \gamma \ } \big(\ca\big)^{\N_0}$ be the group homomorphism which is the composition 
$$   X(A) \inj A^{\N_0} \to \Big(\ca\Big)^{\N_0}.$$
\noindent Let $ E(A) \xrightarrow{\ \eta\ } \big(\ca \big)^{\N_0}$ denote the composition 
$$ E(A) \xrightarrow{\ \pi\ } X(A) \xrightarrow{\ \gamma\ } \Big(\ca\Big)^{\N_0}$$

\vspace{2mm}

\noindent Let $HH_0(E(A)) : =  E(A)/\overline{[E(A),E(A)]}$ where $\overline{[E(A),E(A)]}$ is the closure of the commutator subgroup $[E(A),E(A)]$. The subgroup $[E(A), E(A)]$ is not an ideal of $E(A)$. We then have the following induced maps.
\begin{enumerate}
\item  the Teichm\"uller map  $ \langle \cdot \rangle : A \to HH_0(E(A)).$ 
\item Additive group homomorphism  $ V: HH_0(E(A)) \to HH_0(E(A))$
\item The group homomorphisms which are analogous to the ghost homomorphism 
$ W(A) \xrightarrow{\ \overline{\omega}\ } \Big( \ca \Big)^{\N_0}$.
\begin{enumerate}
\item $HH_0(X(A)) \xrightarrow{ \ \overline{\gamma} \ } \Big(\ca\Big)^{\N_0}$
\item $HH_0(E(A)) \xrightarrow{ \ \overline{\eta} \ } \Big(\ca\Big)^{\N_0}.$
\end{enumerate}
\end{enumerate}
%%%%%%%%%%%%%%%%%%%%%%%%%%%%%%%%%%%%
\begin{remark} Note that if $A$ is commutative $p$-torsion free ring then the ghost map $\overline{\eta}$ is a injective group homomorphism ( See \cite[Corollary 2.10]{cd}) . If $A$ is not commutative then $\overline{\eta}$ and $\overline{\gamma}$ need not be injective even if $\frac{A}{[A,A]}$ is $p$-torsion free ( See  \cite[Theorem 4.2]{hp}). 
\end{remark}
\vspace{2mm}
%%%%%%%%%%%%%%%%%%%%%%%%%%%%%%%%%%%%%%%%%

\begin{remark}[Commutative case] \label{commutative-case} Let $A$ be a unital commutative ring. In this case the two constructions above are identified with the classical construction of the ring of $p$-typical Witt vectors.The discussion after \cite[Corollary  2.10]{cd} establishes an isomorphism $W(A) \to E(A)$, natural in $A$,  given by $ \psi(r) = \sum_{n} V_n\langle r \rangle $. This isomorphism preserves $V$ and $\langle \  \rangle $.
 \end{remark}

\noindent The following is an alternative formulation of Theorem \ref{generalise} as suggested by the referee. We have  two functors $A\mapsto W(A)$ and $A\mapsto HH_0(E(A))$ on the category of unital associative rings. Restricted to the subcategory of commutative rings, these functors are naturally isomorphic (see Remark \ref{commutative-case}).  

\begin{theorem}\label{alternative} There is no natural map from $W(A) \to HH_0(E(A))$ which is a continuous surjective group homomorphism and which induces the natural isomorphism in the commutative case. 
\end{theorem}
\begin{proof}
Let $\phi_A: W(A) \to HH_0(E(A))$ be any natural map which induces the given natural isomorphism in the commutative case. By \eqref{generalise}, It is is enough to show that $\phi_A$ preserves $V$ and $\langle \ \rangle$. For an element $a\in A$ consider the map from $\Z[T] \xrightarrow{f} A$ which sends $T$ to $a$. This gives us the following commutative diagram 

$$\xymatrix{
  W(\Z[T])  \ar[rr]^{W(f)}  \ar[d]^{\phi_{\Z[T]}}  & & W(A)  \ar[d]^{\phi_A}\\
E(\Z[T]) = HH_0(E(\Z[T]))  \ar[rr]^-{HH_0(f)} & & HH_0(E(A)) 
} $$ 
That $\phi_A(\langle a \rangle) = \langle a \rangle$ follows from the fact that the other three maps in the diagram are compatible with $\langle \ \rangle$. To check compatibility of $\phi_A$ with $V$ it is enough to check that for all elements $a\in A$ and $n\in \N_0$ 
$$ \phi_A(V^n\langle a \rangle) = V^n(\langle a \rangle).$$
This also follows from the above diagram. 
\end{proof}

\noindent It is not clear to us if a similar reformulation for Theorem \ref{main}, analogous to \eqref{alternative} can be proved.

%%%%%%%%%%%%%%%%%%%%%%%%%%%%%%%%%%%%%%%%%%%%%
\section{Proof of the Theorem \ref{generalise} and  \ref{main}}
\noindent To prove the Theorem \ref{generalise} we will observe that if there exists a continuous map $W(A) \to HH_0(E(A))$ which commutes with $V$ and $\langle \cdot \rangle$ then it has to commute with the ghost maps $\overline{\eta}$ and $\overline{\omega}$. The argument given here is implicit in the proof of the Theorem 1.3 \cite{hp}. For the convenience it is given below. \\

\begin{lemma}\label{ghost-commute} If there exists a continuous map $\Psi: W(A) \to HH_0(E(A))$ which commutes with $V$ and $\langle \cdot \rangle $  then the following diagram must commute
$$\xymatrix{
 HH_0(E(A)) \ar[r]^-{\overline{\eta}} & \Big( \ca\Big)^{\N_0}\ar@{=}[d]\\
W(A) \ar[u]^-{\Psi}\ar[r]^{\overline{\omega}} & \Big(\ca\Big)^{\N_0}
}$$
\end{lemma}

\begin{proof}  Suppose there exists a continuous group homomorphism $W(A) \to HH_0(E(A))$ satisfying the above mentioned properties. Composing with the natural homomorphism $HH_0(E(A)) \to HH_0(X(A))$ we get a map 
$$\Phi: W(A) \to HH_0(X(A)).$$ 

\noindent To prove the result of the lemma it is thus enough to show that if $\Phi$ exists then it has to commute with the ghost map $\overline{\omega}$ and $\overline{\gamma}$. \\

\noindent By following Hesselholt's  construction in \cite{h2}, we know that the map  
$f: A^{\N_0} \to A^{\N_0}$ given by  
$$\underline{a}:= (a_0,a_1,\cdots) \mapsto (w_0(\underline{a}), w_1(\underline{a}), \cdots )$$ factors thorugh $W(A)$ and we get the following, where $q$ is a surjective map.

$$A^{\N_0}  \xrightarrow{ \ q \ }   W(A) \xrightarrow{ \ \overline{\omega} \ } (\frac{A}{[A,A]})^{\N_0}$$

\noindent Consider the set map $\Omega: A^{\N_0} \to A^{\N_0}$ defined  by 
$$\Omega(\underline{a}) = (\omega_0(\underline{a}), \omega_1(\underline{a}), ...)$$ 
where $\omega_n(\underline{a})  = a_0^{p^n}+pa_1^{p^{n-1}}+\cdots+p^na_n$ are the Witt polynomials.
The Lemma 4.1 in \cite{hp} proves that the image of $\Omega$ is contained in $X(A)$. The fact that both $\Omega$ and $\overline{\omega}\circ q$ are given by the same Witt polynomials, we have the following commutative diagram. 
$$\xymatrix{
   X(A)\ar@{->>}[r]^-{\pi} & HH_0(X(A)) \ar[r]^-{\overline{\gamma}}             & \Big( \ca\Big)^{\N_0}\ar@{=}[d]\\
A^{\N_0}  \ar[r]^q \ar[u]^{\Omega}  & W(A) \ar[r]^{\overline{\omega}} & \Big(\ca\Big)^{\N_0}
}$$

\noindent Suppose there exists a continuous map $\Phi:W(A) \to HH_0(X(A))$. By Lemma 4.2 in \cite{hp}, we know that following diagram is commutative. 
$$\xymatrix{
 X(A)\ar@{->>}[r]^-{\pi} & HH_0(X(A)) \\
A^{\N_0} \ar[r]^q \ar[u]^{\Omega}  & W(A) \ar[u]^-{\Phi}
}$$

\noindent  As $q$ is a surjective map and $ \overline{\gamma} \circ \pi \circ \Omega = \overline{\omega}\circ q$ from the first diagram, this commutative square can be extended to the following commutative diagram.

$$\xymatrix{
   X(A)\ar@{->>}[r]^-{\pi} & HH_0(X(A)) \ar[r]^-{\overline{\gamma}}             & \Big( \ca\Big)^{\N_0}\ar@{=}[d]\\
A^{\N_0}  \ar[r]^q \ar[u]^{\Omega}  & W(A) \ar[u]^-{\Phi}\ar[r]^{\overline{\omega}} & \Big(\ca\Big)^{\N_0}
}$$

\noindent This proves that a continuous map $\Phi: W(A) \to HH_0(X(A))$ which commutes with $V$ and $\langle \cdot \rangle$  has to commute with the ghost map $\overline{\omega}$ and $\overline{\gamma}$. Thus 
it will commute with the ghost maps $\overline{\omega}$ and $\overline{\eta}$.

\end{proof}
%%%%%%%%%%%%%%%%%%%%%%%%%%%%%%%%%%%%%%%%%%%%%%

\begin{proof}[Proof of Theorem \ref{generalise}] We will show that there does not exist a continuous surjective map \\
$\Phi: W(A) \to HH_0(E(A))$ which commutes with $V$ and $\langle \cdot \rangle$. As explained in Lemma \ref{ghost-commute}, it is enough to show that there does not exist a continuos surjective map 
$W(A) \to HH_0(X(A))$ which commutes with the ghost maps.  \\

\noindent Let $p$ be any prime number and $ A = \Z\{X,Y\}$. Let $\langle X \rangle \langle Y \rangle \in  HH_0(X(A))$ and let $\alpha : = (\alpha_1, \alpha_2, \cdots ) \in W(A)$ such that $\Phi(\alpha) = \langle X \rangle \langle Y \rangle $.

\begin{align*}
\overline{\gamma}(\langle X \rangle \langle Y \rangle ) & = \overline{\omega}(\alpha)  \ \ \ \ \ \ \ \ \ \ \ \ \ \ \ \hspace{12mm}   \cdots\cdots   (\text{ By \ Lemma \ \ref{ghost-commute} }) \\
                                     & =  (\overline{\alpha_1}, \ \overline{\alpha_1}^p + p\overline{\alpha_2}, \ \overline{\alpha_1}^{p^2}+p\overline{\alpha_2}^{p} + p^2\overline{\alpha_3}, \cdots  ) \\
                                     & = (\overline{\alpha_1},\overline{\alpha_1}^p, \ \cdots ) \  (\ {\rm mod} \ pA) 
 \end{align*}
 
This gives us, 

$$\overline{XY} =  \overline{\alpha}_1 \ ({\rm mod} \ pA),\  \overline{X^pY^p} = \overline{\alpha_1}^p = \overline{XY}^p \ ({\rm mod } \ pA), \ \cdots $$ 

\vspace{2mm}

In the next Lemma, we will show that the equality $\overline{X}^p\overline{Y}^p = \overline{XY}^p \ ({\rm mod } \ pA)$ is not possible. 
\end{proof}
%%%%%%%%%%%%%%%%%%%%%%%%%%%%%%%%%%%%%%%%%%%%

\begin{lemma}\label{p-power} Suppose $p$ is any prime number and $A = \Z\{X, Y\}$. Let $\overline{A} : = \Z/p\Z\{X,Y \}$. Then $X^pY^p \neq (XY)^p ({\rm mod} \  [\overline{A}, \overline{A}])$
\end{lemma}
\begin{proof} It is enough to find a homomorphism $f$ from $\overline{A}$ to another ring $B$ such that  $f(X)^pf(Y)^p \neq f(XY)^p   ({\rm mod} \  [B, B])$. \\

\noindent  Let $ B : = M_2(\F_p)$. Consider a homomorphism 
$$f: \overline{A} \to M_2(\F_p))$$ \\
\[ f(X) = R :=  \left( \begin{array}{ccc}               
0 & 0 \\
1 & 0  \\
\end{array}
\right) \]

\[ f(Y) =  S: =  \left( \begin{array}{ccc}               
0 & 1\\
0 & 0  \\
\end{array}
\right) \]
 
 \vspace{2mm}
 
\noindent If $X^pY^p - (XY)^p \in [\overline{A}, \overline{A}]$ then 
$ R^pS^p - (RS)^p   \in [M_2(\F_p), M_2(\F_p)]$.  This will imply that 
$$Tr (R^pS^p-(RS)^p ) = 0.$$ \\
\noindent Now $R^p = S^p = 0$ and $Tr(R^pS^p - (RS)^p) ) = -1$. This implies that 
$X^pY^p - (XY)^p \not\in [\overline{A}, \overline{A}]$. \\ 

\end{proof}

\begin{remark} The above proof of the claim that $X^pY^p - (XY)^p  \not\in [\overline{A}, \overline{A}]$ simplifies the arguments of the main Theorem 2.1 of \cite{hp} for $p=2$ and generalises it to any prime number $p$.
\end{remark}
%%%%%%%%%%%%%%%%%%%%%%%%%%%%%%%%%%%%%%%%%%%%%%
\noindent We will prove below the Theorem \ref{main} by using the Lemma \ref{p-power}.
\begin{proof}[Proof of Theorem \ref{main}]
Suppose $p$ is any prime number, $A = \Z\{X,Y\}$ and $\overline{A} := \Z/p\Z\{X,Y\}$. Suppose there exists a map 
$\rho : HH_0(E(A)) \to W(A) $ which commutes with the ghost maps i.e  $\overline{\omega} \circ \rho = \overline{\eta}$.  Consider the element $ \rho(\langle X \rangle \langle Y \rangle ) = (\alpha_1, \alpha_2, \cdots ) \in W(A)$.  Thus we have the equality, 

$$\overline{\omega} \circ \rho(\langle X\rangle \langle Y \rangle ) = \overline{\eta}(\langle X \rangle \langle Y \rangle) $$

\begin{align*}
\overline{\omega}( \rho(\langle X \rangle \langle Y \rangle ) ) & = (\alpha_1, \alpha_1^p + p \alpha_2, \alpha_1^{p^2}+p\alpha_2^p + p^2 \alpha_3, \  \cdots )  ({\rm mod} \ [A, A])\\
& = (\alpha_1, \alpha_1^p, \alpha_1^{p^2}, \ \cdots ) ({\rm mod} \ [\overline{A}, \overline{A}])
\end{align*}

\noindent We also have,

\begin{align*}
\eta(\langle X \rangle \langle Y \rangle) & = ( XY, X^pY^p, X^{p^2}Y^{p^2}, \ \cdots )  ({\rm mod} \ [A, A])\\
                                            & = ( XY, X^pY^p, X^{p^2}Y^{p^2}, \ \cdots )  ({\rm mod} \ [\overline{A}, \overline{A}])
\end{align*}

\noindent Thus, $\alpha_1^p = (XY)^p = X^pY^p$. This is not possible by the Lemma \ref{p-power}. Thus, there does not exist any  map $\psi: HH_0(E(A)) \to W(A)$ which commutes with the ghost maps. 

\end{proof}

%%%%%%%%%%%%%%%%%%%%%%%%%%%%%%%%%%%%%%%%%%%%%%

\end{document}